\documentclass[12pt]{amsproc}
\usepackage{amssymb,hyperref,tikz-cd,amsaddr}
\usepackage[style=numeric]{biblatex}
\usepackage[shortlabels]{enumitem}
\renewbibmacro*{doi+eprint+url}{
  \iftoggle{bbx:doi}
    {\iffieldundef{url}{\printfield{doi}}{}}
    {}
  \newunit\newblock{}
  \iftoggle{bbx:eprint}
    {\usebibmacro{eprint}}
    {}
  \newunit\newblock{}
  \iftoggle{bbx:url}
    {\usebibmacro{url+urldate}}
    {}
	} 

\newtheorem{theorem}{Theorem}[section]
\newtheorem{lemma}[theorem]{Lemma}

\newtheorem{corollary}[theorem]{Corollary}
\theoremstyle{definition}
\newtheorem{definition}[theorem]{Definition}
\theoremstyle{remark}

\newtheorem{example}[theorem]{Example}
\newtheorem{reduction}{Reduction}

\renewcommand{\hat}{\widehat}

\DeclareMathOperator{\GL}{GL}

\DeclareMathOperator\Ind{Ind}

\DeclareMathOperator\id{id}
\DeclareMathOperator\tr{trace}

\DeclareMathOperator{\kron}{Kr}
\DeclareMathOperator{\conj}{Conj}
\DeclareMathOperator{\rconj}{RConj}
\DeclareMathOperator{\St}{St}
\DeclareMathOperator{\Stab}{Stab}

\newcommand\CC{\mathbf{C}}
\newcommand\Fq{\mathbf{F}_q}
\newcommand\ii{\mathbf{i}}
\newcommand\jj{\mathbf{j}}
\newcommand\kk{\mathbf{k}}

\newcommand\gen[1]{\langle#1\rangle}

\theoremstyle{remark}
\newtheorem{remark}[theorem]{Remark}

\numberwithin{equation}{section}
\newcommand{\abs}[1]{\lvert#1\rvert}
\title[Kronecker coefficients and simultaneous conjugacy classes]{Kronecker Coefficients and\\Simultaneous Conjugacy Classes}
\dedicatory{Dedicated to the memory of A.~P.~Balachandran.}
\author[Ganguly, Paul, Prasad, Raghavan, Velmurugan]{Jyotirmoy Ganguly$^1$, Digjoy Paul$^6$\\Amritanshu Prasad$^{3,5}$, K N Raghavan$^4$, Velmurugan S$^{2}$}
\address{$^1$GITAM University, Bengaluru, India.\\$^2$Indian Institute of Science, Bengaluru, India.\\$^3$The Institute of Mathematical Sciences, Chennai, India.\\$^4$Krea University, Sri City, India.\\$^5$Homi Bhabha National Institute, Mumbai, India.\\$^6$SRM University-AP, Amravati, India.\\}

\begin{document}
\begin{abstract}
	A Kronecker coefficient is the multiplicity of an irreducible representation of a finite group $G$ in a tensor product of irreducible representations.
	We define Kronecker Hecke algebras and use them as a tool to study Kronecker coefficients in finite groups.
	We show that the number of simultaneous conjugacy classes in a finite group $G$ is equal to the sum of squares of Kronecker coefficients, and the number of simultaneous conjugacy classes that are closed under elementwise inversion is the sum of Kronecker coefficients weighted by Frobenius-Schur indicators.
	We use these tools to investigate which finite groups have multiplicity-free tensor products.
	We introduce the class of doubly real groups, and show that they are precisely the real groups which have multiplicity-free tensor products.
	We show that non-Abelian groups of odd order, non-Abelian finite simple groups, and most finite general linear groups do not have multiplicity-free tensor products.
\end{abstract}

\maketitle
\section{Introduction}\label{section:introduction}
Let $G$ be a finite group.
In this article all representations are assumed to be over the complex numbers.
Given irreducible representations $U_1,\dotsc,U_{d+1}$ of $G$, define the \emph{generalized Kronecker coefficient} $\kappa(U_1,\dotsc,U_{d+1})$ to be the multiplicity of the trivial representation of $G$ in $U_1\otimes \dotsb \otimes U_{d+1}$.
If $W'$ denotes the contragredient of a representation $W$ of $G$, $\kappa(U,V,W')$ is the multiplicity of $W$ in $U\otimes V$.
Known by the name of Clebsch-Gordan multiplicities in the context of Lie groups, these numbers were studied by Murnaghan~\cite{MR1507347} for the symmetric group, who called them Kronecker coefficients.

Let $G$ act on $G^d$ by simultaneous conjugation:
\begin{equation}\label{eq:simultaneous-conjugation}
	g\cdot (g_1,\dotsc,g_d) = (gg_1g^{-1},\dotsc,gg_dg^{-1}).
\end{equation}
The orbits of this action are called \emph{simultaneous conjugacy classes}.
A simultaneous conjugacy class in $G^d$ is said to be \emph{real} if it is closed under componentwise inversion.
Denote the set of simultaneous conjugacy classes in $G^d$ by $\conj_d(G)$ and the set of real simultaneous conjugacy classes by $\rconj_d(G)$.

In this article, we show that the number of simultaneous conjugacy classes in $G^d$ is related to the Kronecker coefficients of $G$.
\begin{align}\label{eq:A}\tag{A}
	|\conj_d(G)| & = \sum_{(V_1,\dotsc,V_{d+1})} \kappa(V_1,\dotsc,V_{d+1})^2,\\\tag{B}\label{eq:B}
	|\rconj_d(G)| & = \sum_{(V_1,\dotsc,V_{d+1})} \sigma(V_1)\dotsb \sigma(V_{d+1})\kappa(V_1,\dotsc,V_{d+1}).
\end{align}
The sums are over all $(d+1)$-tuples of irreducible representations of $G$, and for an irreducible representation $U$ of $G$, $\sigma(U)$ denotes its Frobenius-Schur indicator.
In the context of symmetric groups, these results were discovered by Ben Geloun and Ramgoolam~\cite{MR4591600}.
Eq.~\eqref{eq:B} incorporates subtleties that arise from the Frobenius-Schur indicator absent in the case of symmetric groups.

The main tool is the Kronecker-Hecke algebra $\kron_d(G)$ of $G$, introduced in Section~\ref{section:kronecker-hecke-algebra},
which is a group Hecke algebra of $G^{d+1}$ with respect to the diagonal subgroup $\Delta G$.
Indeed, Eq.~\eqref{eq:A} is a consequence of the Wedderburn decomposition of $\kron_d(G)$ (Theorem~\ref{theorem:sum-of-squares}).
Eq.~\eqref{eq:B} is an application of a theorem of Frame on Hecke algebras (see Theorem~\ref{theorem:trace-dagger-kron}).

Given a group $G$ and a subgroup $K$, $(G,K)$ is called a \emph{Gelfand pair} if the Hecke algebra $H(G,K)$ is commutative.
Sometimes, this commutativity follows from Gelfand's trick: if $KgK = Kg^{-1}K$ for all $g\in G$, then $H(G,K)$ is commutative.
In this case $(G,K)$ is called a \emph{symmetric Gelfand pair}~\cite[Section~4.3]{MR2389056}.
Theorem~\ref{theorem:easy-gelfand} is a result of independent interest: we show that $(G,K)$ is a symmetric Gelfand pair if and only if the space $V^K$ of $K$-invariant vectors is at most one-dimensional for every irreducible representation $V$ of $G$, and if $V^K\neq 0$, then the Frobenius-Schur indicator $\sigma(V)=1$.

In the context of Kronecker-Hecke algebras, $(G^{d+1},\Delta G)$ is a Gelfand pair if and only if it has multiplicity-free $d$-fold tensor products (Theorem~\ref{theorem:commutativity-kronecker-hecke}), i.e.,
\begin{equation}\label{eq:commutativity-kronecker-hecke}
	\kappa(V_1,\dotsc,V_{d+1})\leq 1 \text{ for all }(V_1,\dotsc,V_{d+1})\in \hat G^{d+1}.
\end{equation}
On the other hand, $(G^{d+1}, \Delta G)$ is a symmetric Gelfand pair if and only if $G$ is $d$-real (Theorem~\ref{theorem:commutativity-kronecker-hecke}), which means that every $d$-fold simultaneous conjugacy class in $G^d$ contains its coordinatewise inverse (Definition~\ref{definition:d-real}).

Burnside's lemma gives a formula~\eqref{eq:burnside} for the number of simultaneous conjugacy classes in $G^d$.
Eq.~\eqref{eq:B} gives rise to a formula (Theorem~\ref{theorem:rational-gf}) for the number of real simultaneous conjugacy classes in $G^d$:
\begin{equation}\label{eq:C}\tag{C}
	|\rconj_d(G)| = \frac 1{|G|}\sum_{g\in G} r(g)^{d+1},	
\end{equation}
where $r(g)$ is the number of square roots of $g$ in $G$.
For example, for the monster group $M$, one easily computes $|\rconj_2(M)|=240440865730496103575552476238$ using its character table in GAP.

When the point $g$ is chosen uniformly at random from the finite group $G$, the number $r(g)$ of square roots of $g$ in $G$ is a random variable.
The formula~\eqref{eq:C} has the following probabilistic interpretation:
\begin{quote}
	The number of real simultaneous conjugacy classes in $G^d$ is the $(d+1)$st moment of $r(g)$.
	In particular, for groups $G$ and $H$, $|\rconj_d(G)| = |\rconj_d(H)|$ for all $d\geq 1$ if and only if $r(g)$ has the same probability distribution for $G$ and $H$.
\end{quote}

It turns out that groups for which $d$-fold tensor products are multiplicity-free for $d\geq 3$ are Abelian (Theorem~\ref{theorem:higher-hecke-commutative}).
Hence we focus on the case where $d=2$.
A group has multiplicity-free $2$-fold tensor products if the tensor product $U\otimes V$ of any two irreducible representations $U$ and $V$ is multiplicity-free.
We refer to them as \emph{groups with multiplicity-free tensor products}.

A group is said to be \emph{doubly real} if every pair of elements is simultaneously conjugate to its inverse.
Theorem~\ref{theorem:d-real} implies that a doubly real group has multiplicity-free tensor products. 
The converse of this result holds for real groups: every real group with multiplicity-free tensor products is doubly real (Theorem~\ref{theorem:real-multiplcity-free}).

In Section~\ref{section:doubly-real}, we show that dihedral groups and extraspecial $2$-groups are doubly real.
Generalized quaternion groups are doubly real if and only if they are real.
Heisenberg groups over finite fields are doubly real if and only if the underlying field has order $2$.

In Section~\ref{sec:groups-that-dont} we exhibit several classes of groups that do not have multiplicity-free tensor products.
Non-Abelian groups of odd order do not have multiplicity-free tensor products (Theorem~\ref{theorem:odd-order}).
Symmetric groups $S_n$ for $n\geq 5$ and alternating groups $A_n$ for $n\geq 4$ do not have multiplicity-free tensor products.
In fact, no non-Abelian finite simple group has multiplicity-free tensor products (Theorem~\ref{theorem:finite-simple}).

These findings are summarized in Figure~\ref{fig:summary}, where MFTP stands for groups having multiplicity-free tensor products and NMFTP stands for groups without multiplicity-free tensor products.

\begin{figure}
	\begin{center}
		\scalebox{0.9}{
		\begin{tikzpicture}[font=\footnotesize]
			\path (3,0) node [draw, align=left] (P1) { Groups};
			\path (-3,-1.5) node [draw, align=left] (P21) {MFTP};
			\path (3,-1.5) node [draw, align=left] (P22) {Real};
			\path (9,-1.5) node [draw, align=left] (P23) {NMFTP};
			\path (-4.5,-3) node [draw, align=left] (P31) {Abelian};
			\path (-2.5,-3) node [draw, align=left] (P32) {$Q(A)$};
			\path (0,-3) node [draw, align=left] (P33) {Doubly Real};
			\path (6,-4.2) node [draw, align=left] (P34) {Non-Abelian\\Simple};
			\path (3.9,-6.5) node [draw, align=left] (P35) {Simple Real\\except $C_2$};
			\path (6.7,-6.5) node [draw, align=left] (P37) {Non-Abelian\\ odd order};
			\path (9.1,-6.5) node [draw, align=left] (P38) {$H_n(\mathbb{F}_q)$\\ $q>2$};
			\path (11.2,-6.24) node [draw, align=left] (P50) {$GL_n(\mathbf F_q)$\\except\\$n=q=2$};
			\path (-1.5,-6.5) node [draw, align=left] (P42) {Extraspecial\\ $2$ groups};
			\path (1.3,-6.5) node [draw, align=left] (P43) {Generalized\\Dihedral};
			\path (5.5,-2.8) node [draw, align=left] (P39) {$S_n$\\$n\geq 5$};
			\path (-3.9,-6.5) node [draw, align=left] (P46) {Real \\$Q(A)$};
			\draw
			(P1) edge (P21)
			(P1) edge (P22)
			(P21) edge (P31)
			(P21) edge (P32)
			(P22) edge (P33)
			(P22) edge (P35)
			(P34) edge (P35)
			(P23) edge (P37)
			(P23) edge (P38)
			(P1) edge (P23)
			(P23) edge (P34)
			(P33) edge (P42)
			(P33) edge (P43)
			(P23) edge (P39)
			(P21) edge (P33)
			(P22) edge (P39)
			(P33) edge (P46)
			(P23) edge (P50)
			(P32) edge (P46);
		\end{tikzpicture}
		}
	\end{center}
	\caption{A summary of our findings on groups with multiplicity-free tensor products and doubly real groups.}\label{fig:summary}
\end{figure}
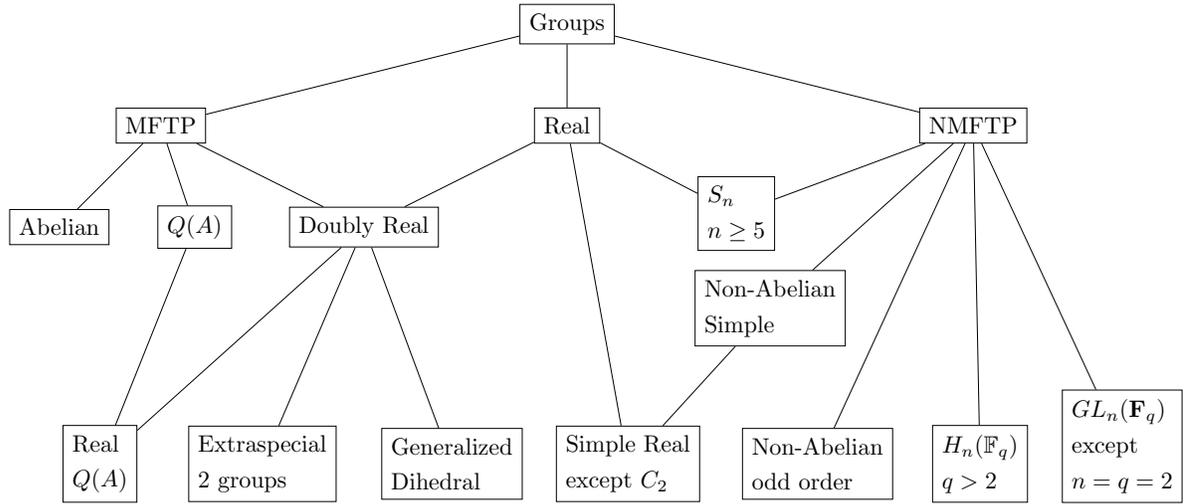

\section{The Kronecker Hecke Algebra}\label{section:kronecker-hecke-algebra}
\subsection{Group Hecke Algebras}
The results of this section are well-known (see, e.g.,~\cite{MR2389056}).
For the reader's convenience, we include their proofs, which are quite short.
Recall that the group algebra $\CC[G]$ of a finite group $G$ is the complex vector space with basis $\{1_g\mid g\in G\}$ endowed with multiplication defined by bilinearly extending the rule $1_g1_h = 1_{gh}$ for all $g,h\in G$.
For a subgroup $K$ of $G$, let
\begin{displaymath}
	\epsilon_K = \frac 1{\abs K}\sum_{k\in K} 1_k.
\end{displaymath}
Then $\epsilon_K$ is idempotent in $\CC[G]$.
The Hecke algebra of $G$ with respect to $K$ is defined as
\begin{displaymath}
	H(G,K) = \epsilon_K\CC[G]\epsilon_K.
\end{displaymath}

Let $\hat G$ denote a set of representatives of the isomorphism classes of irreducible representations of $G$.
For a representation $V$ of $G$ let $V'$ denote its contragredient (viewed as a right $\CC[G]$-module).
The Wedderburn decomposition
\begin{equation}\label{eq:wedderburn}
	\CC[G] \cong \bigoplus_{V\in \hat G} V\otimes V',
\end{equation}
is an isomorphism of $(\CC[G], \CC[G])$-bimodules.
Let $V^K=\epsilon_K V$, the subspace of $K$-invariant vectors in $V$.
Eq.~\eqref{eq:wedderburn} gives the Wedderburn decomposition of $H(G,K)$ as
\begin{equation}\label{eq:wedderburn-hecke}
	H(G,K) \cong \bigoplus_{V\in \hat G} V^K\otimes (V^K)'.
\end{equation}
The orbits for the action of $K\times K$ on $G$ by $(k_1,k_2)g = k_1gk_2^{-1}$ are called \emph{double cosets} of $K$ in $G$.
Let $K\backslash G/K$ denote the set of double cosets of $K$ in $G$.
\begin{theorem}\label{theorem:dimension-hecke-algebra}
	Let $G$ be a finite group and $K$ be a subgroup.
	Then
	\begin{equation}\label{eq:dimension-hecke-algebra}
		\sum_{V\in \hat G} \dim {(V^K)}^2 = \abs{K\backslash G/K}.
	\end{equation}
\end{theorem}
\begin{proof}
	The Hecke algebra $H(G,K)$ consists of elements of the form $\sum_{g\in G} a_g 1_g$ such that the function $g\mapsto a_g$ is constant on double cosets of $K$ in $G$.
	Therefore, its dimension is $\abs{K\backslash G/K}$.
	On the other hand, by the Wedderburn decomposition~\eqref{eq:wedderburn-hecke}, the dimension is $\sum_{V\in \hat G} {\dim(V^K)}^2$.
\end{proof}
\begin{theorem}\label{theorem:commutativity}
	Let $G$ be a finite group and $K$ be a subgroup.
	\begin{enumerate}[1.]
		\item The algebra $H(G,K)$ is commutative if and only if $\dim V^K\leq 1$ for every $V\in \hat G$.
		\item If $KgK = Kg^{-1}K$ for all $g\in G$, then $H(G,K)$ is commutative (Gelfand's Trick).
	\end{enumerate}
\end{theorem}
\begin{proof}
	The semisimple algebra $H(G,K)$ is the endomorphism algebra of the right regular $H(G,K)$-module.
	Hence, it is commutative if and only if this module has a multiplicity-free decomposition into simples.
	From the Wedderburn decomposition~\eqref{eq:wedderburn-hecke}, this is the case if and only if $\dim V^K\leq 1$ for every $V\in \hat G$.

	For each $g\in G$, let $1_g^\dagger = 1_{g^{-1}}$.
	By linear extension, $\dagger$ extends to an anti-involution on $\CC[G]$.
	Since $\epsilon_K^\dagger = \epsilon_K$, $\dagger$ restricts to an anti-involution on $H(G,K)$.
	If $KgK = Kg^{-1}K$ for all $g\in G$, then the restriction of $\dagger$ to $H(G,K)$ is the identity.
	If the identity map is an anti-involution, the algebra $H(G,K)$ has to be commutative.
\end{proof}
\subsection{Kronecker Hecke Algebra and Simultaneous Conjugacy}\label{sec:kronecker-hecke}
\begin{definition}[Kronecker Hecke algebra]\label{defn:kha}
For every finite group $G$ and every positive integer $d$, the $d$th \emph{Kronecker Hecke algebra} of $G$ is defined as 
\begin{displaymath}
	\kron_d(G) = H(G^{d+1},\Delta G).
\end{displaymath}
\end{definition}
The group $G$ acts on $G^d$ by \emph{simultaneous conjugation}, i.e., for $g\in G$ and $(g_1,\dotsc,g_d)\in G^d$,
\begin{displaymath}
	g\cdot (g_1,\dotsc,g_d) = (gg_1g^{-1},\dotsc,gg_d{g}^{-1}).
\end{displaymath}
\begin{definition}[Simultaneous conjugacy classes]\label{definition:simultaneous-conjugacy}
	The orbits of the action of $G$ on $G^d$ by simultaneous conjugation are called \emph{simultaneous conjugacy classes}.
	The set of simultaneous conjugacy classes in $G^d$ is denoted by $\conj_d(G)$.
\end{definition}

\begin{lemma}\label{lemma:bijection}
	The map $(g_1,\dotsc,g_d)\mapsto (g_1,\dotsc,g_d,1)$ gives rise to a bijection
	\begin{displaymath}
		\conj_d(G)\to \Delta G\backslash G^{d+1}/\Delta G.
	\end{displaymath}
\end{lemma}
\begin{proof}
	It is easy to check that the map gives rise to a well-defined function whose inverse is induced by the map $(g_1,\dotsc,g_{d+1})\mapsto (g_1 g_{d+1}^{-1},\dotsc,g_d g_{d+1}^{-1})$.
\end{proof}
\begin{theorem}\label{theorem:sum-of-squares}
	Let $G$ be a finite group and $d$ be a positive integer.
	Then
	\begin{displaymath}
		\sum_{(V_1,\dotsc,V_{d+1})\in \hat G^{d+1}} \kappa(V_1,\dotsc,V_{d+1})^2 = \abs{\conj_d(G)}.
	\end{displaymath}
\end{theorem}
\begin{proof}
	The theorem follows from Lemma~\ref{lemma:bijection} and Eq.~\eqref{eq:dimension-hecke-algebra}.
	For $\kron_d(G)$, the right-hand side of Eq.~\eqref{eq:dimension-hecke-algebra} is $\abs{\Delta G\backslash G^{d+1}/\Delta G}$ which, by Lemma~\ref{lemma:bijection}, is the same as $\abs{\conj_d(G)}$.
	The irreducible representations of $G^{d+1}$ are of the form $V_1\otimes\dotsb\otimes V_{d+1}$ where $(V_1,\dotsb,V_{d+1})\in \hat G^{d+1}$.
	Since $\dim (V_1\otimes\dotsb\otimes V_{d+1})^{\Delta G} = \kappa(V_1,\dotsc,V_{d+1})$, the left-hand side of~\eqref{eq:dimension-hecke-algebra} becomes $\sum \kappa(V_1,\dotsc,V_{d+1})^2$.
\end{proof}
In the group algebra $\CC[G^d]$, the subspace
\begin{displaymath}
	\CC[G^d/\sim]:=\Big\{\sum_{(g_1,\dotsc,g_d)\in G^d} a_{g_1,\dotsc,g_d} 1_{(g_1,\dotsc g_d)}\mid a_{g_1,\dotsc,g_d} = a_{g\cdot (g_1,\dotsc,g_d)}\text{ for all }g, g_1,\dotsc,g_d\in G\Big\}
\end{displaymath}
is a subalgebra.
\begin{theorem}\label{theorem:conj-kron}
	For every finite group $G$ and every positive integer $d$, the subalgebra $\CC[G^d/\sim]$ of $\CC[G^d]$ is isomorphic to the Kronecker-Hecke algebra $\kron_d(G)$.
\end{theorem}
\begin{proof}
The map $1_{(g_1,\dotsc,g_d)}\mapsto 1_{(g_1,\dotsc,g_d,1)}$ is an algebra embedding of $\CC[G^d]$ in $\CC[G^{d+1}]$.
By Lemma~\ref{lemma:bijection}, it induces an isomorphism $\CC[G^d/\sim]\cong \kron_d(G)$.
\end{proof}
\begin{remark}
	When $d=1$, $\CC[G/\sim]$ is the centre of $\CC[G]$.
	Theorem~\ref{theorem:sum-of-squares} reduces to the familiar result that the number of isomorphism classes of irreducible representations of $G$ is equal to the number of conjugacy classes of $G$.
\end{remark}
\subsection{Kronecker Hecke Algebras and Multiplicity-Free Tensor Products}\label{section:multiplicity-free}
\begin{definition}
	Let $G$ be a finite group and $d$ be a positive integer.
	Then $G$ is said to have multiplicity-free $d$-fold tensor products if, for all irreducible representations $V_1,\dotsc,V_d$ of $G$, the representation $V_1\otimes\dotsb\otimes V_d$ has a multiplicity-free decomposition into irreducible representations of $G$.
\end{definition}

The condition that $G$ has multiplicity-free $2$-fold tensor products is equivalent to the condition that $U\otimes V$ has a multiplicity-free decomposition into irreducible representations for all $U,V\in \hat G$.
In this case, we say that $G$ has multiplicity-free tensor products.
\begin{definition}\label{definition:d-real}
	A simultaneous conjugacy class in $G^d$ is said to be real if it is closed under componentwise inversion.
	The group $G$ is said to be $d$-real if every simultaneous conjugacy class in $G^d$ is real.
\end{definition}
\begin{theorem}\label{theorem:commutativity-kronecker-hecke}
	Let $G$ be a finite group and $d$ be a positive integer.
	\begin{enumerate}[1.]
		\item Then $G$ has multiplicity-free $d$-fold tensor products if and only if $\kron_d(G)$ is commutative.
		\item If $G$ is $d$-real then $G$ has multiplicity-free $d$-fold tensor products.
	\end{enumerate}
\end{theorem}
\begin{proof}
	The first assertion follows immediately from the first assertion of Theorem~\ref{theorem:commutativity}.
	For the second assertion, note that if $G$ is $d$-real, then for all $(g_1,\dotsc,g_{d+1})\in G^{d+1}$ there exists $g\in G$ such that
	\begin{displaymath}
		g g_i g_{d+1}^{-1} g^{-1} = g_{d+1}g_i^{-1} \text{ for }i=1,\dotsc,d.
	\end{displaymath}
	We have
	\begin{align*}
		\Delta(g_{d+1}^{-1}g)(g_1,\dotsc,g_d,g_{d+1})\Delta(g_{d+1}^{-1}g^{-1}) & = \Delta(g_{d+1}^{-1})(gg_1g_{d+1}^{-1}g^{-1},\dotsc,gg_dg_{d+1}^{-1}g^{-1},1)\\
		& = \Delta(g_{d+1}^{-1})(g_{d+1}g_1^{-1},\dotsc,g_{d+1}g_d^{-1},1)\\
		& = (g_1^{-1},\dotsc,g_d^{-1},g_{d+1}^{-1}).
	\end{align*}
	Hence, $\Delta(G)(g_1,\dotsc,g_{d+1})\Delta(G) = \Delta(G)(g_1^{-1},\dotsc,g_d^{-1},g_{d+1}^{-1})\Delta(G)$ for all $(g_1,\dotsc,g_{d+1})\in G^{d+1}$.
	By the second assertion of Theorem~\ref{theorem:commutativity}, $\kron_d(G)$ is commutative.
\end{proof}
\begin{remark}
	For any finite group $G$, the Kronecker-Hecke algebra $\kron_1(G)$ is commutative.
	This follows from the fact that $\kappa(U,V)=\delta_{UV'}$ for all $U,V\in \hat G$.
\end{remark}
\begin{theorem}\label{theorem:higher-hecke-commutative}
	Any finite group $G$ with multiplicity-free $3$-fold tensor products is commutative.
\end{theorem}
\begin{proof}
	If $G$ is non-commutative, it has an irreducible representation $V$ of dimension at least $2$.
	Let $V'$ denote the contragredient representation of $V$.
	The trivial representation occurs in $V\otimes V'$, so $V\otimes V'=1\oplus W$ for some representation $W$ of $V$.
	Now
	\begin{displaymath}
		V\otimes V'\otimes V'\otimes V = (1\oplus W)\otimes (1\oplus W') = 1\oplus(W\oplus W')\oplus (W\otimes W'). 
	\end{displaymath}
	Since $W\otimes W'$ contains the trivial representation, we must have $\kappa(V,V',V',V)\geq 2$.
\end{proof}
\section{Real Simultaneous Conjugacy Classes and Kronecker Coefficients}\label{section:real-simulateneous-conjugacy}
\subsection{Frobenius-Schur Indicators and Self-Inverse Double Cosets}\label{sec:appl-hecke-algebr}
The Frobenius--Schur~\cite{FS} indicator $\sigma(V)$ of an irreducible representation $V$ of $G$ is defined as
\begin{displaymath}
    \sigma(V) = \frac 1{|G|}\sum_{g\in G} \chi_V(g^2).
\end{displaymath}
If $V$ is self-contragredient, then $\sigma(V)\in \{1,-1\}$, and otherwise $\sigma(V)=0$.
For self-contragredient representations $\sigma(V)=1$ (resp. $\sigma(V)=-1$) if and only if $V$ admits a non-degenerate $G$-invariant symmetric (resp. alternating) bilinear form.

The following result is due to Frame~\cite[Theorem~A]{MR4027}.
\begin{theorem}\label{theorem:trace-dagger-hecke}
	Let $G$ be a group and $K$ any subgroup.
	Then
	\begin{displaymath}
		\sum_{V\in \hat G} \sigma(V)\dim(V^K) = \#\{KgK\in K\backslash G/K \mid KgK=Kg^{-1}K\}.
	\end{displaymath}
\end{theorem}
Frame easily deduces Theorem~\ref{theorem:trace-dagger-hecke} using the following lemma~\cite[Theorem~3.1]{MR4027}:
\begin{lemma}
    We have
    \begin{equation}\label{eq:frame-3.1}
        \#\{KxK\mid x^{-1}\in K x K\} = \frac 1{|G|} \#\{(Kx,g) \mid xg^2\in Kx\}.
    \end{equation}
\end{lemma}
We give a simpler proof Frame's lemma.
\begin{proof}
	    The function $(Kx,Ky)\mapsto Kyx^{-1}K$ induces a bijection:
    \begin{equation}\label{eq:double-coset-bijection}
        (K\backslash G \times K\backslash G)/\Delta G \xrightarrow{\sim} K\backslash G /K,
    \end{equation}
    A double coset $KxK$ is said to be self-inverse if $x^{-1}\in KxK$.
    A pair $(Kx, Ky)$ is said to be symmetric if $(Ky, Kx)$ lies in the same $\Delta G$-orbit as $(Kx, Ky)$.
    Let $X\subset K\backslash G \times K\backslash G$ be the set of symmetric pairs.
    The bijection~\eqref{eq:double-coset-bijection} restricts to a bijection between $X/\Delta G$ and the set of self-inverse double cosets.

    Hence, by Burnside's lemma, the left side of~\eqref{eq:frame-3.1} can be expressed as
    \begin{align*}
        \frac 1{|G|}\sum_{g\in G} |X^g| &= \frac 1{|G|} \sum_{g\in G}\quad \sum_{\{(Kx, Ky)\mid \exists u\in G, Kxu=Ky=Kyg, Kyu=Kx=Kxg\}} 1\\
        & = \frac 1{|G|} \sum_{\{(Kx, Ky, u, g) \mid Kxu=Ky=Kyg, Kyu=Kx=Kxg \}} \frac 1{|\Stab_G(Kx,Ky)|}\\
        & = \frac 1{|G|} \sum_{\{(Kx, u) \mid Kxu^2 = Kx\}} \sum_{\{g\mid Kxg = Kx, Kxug = Kxu\} } \frac 1{|\Stab_G(Kx, Kxu)|}.
    \end{align*}
    The inner sum in the last step evaluates to $1$, giving the right-hand side of~\eqref{eq:frame-3.1}.
\end{proof}
\begin{proof}[Frame's proof of Theorem~\ref{theorem:trace-dagger-hecke}]
	For any $g\in G$, we have
	\begin{displaymath}
		\tr(g^2; \CC[K\backslash G]) = \#\{Kx\in K\backslash G \mid xg^2 \in  Kx\}.
	\end{displaymath}
	Averaging over $g\in G$, and applying~\eqref{eq:frame-3.1}, we get
	\begin{align}\label{eq:average-trace-1}
		\frac 1{|G|}\sum_{g\in G}\tr(g^2;\CC[K\backslash G]) &= \frac 1{|G|}\#\{(Kx, g) \mid xg^2 \in  Kx\} \\
		\nonumber & = \#\{KxK\mid x^{-1} \in KxK\}
	\end{align}
	On the other hand the Wedderburn decomposition of $G[G]$ yields
	\begin{displaymath}
		\CC[K\backslash G] \cong \bigoplus_{V\in \hat G}  V \otimes V^K,
	\end{displaymath}
	so that
	\begin{displaymath}
		\tr(g^2;\CC[K\backslash G]) = \sum_{V\in \hat G} \dim V^K\chi_V(g^2).
	\end{displaymath}
	Averaging over $g\in G$ and using the Frobenius-Schur indicator, we get
	\begin{equation}\label{eq:average-trace-2}
		\frac 1{|G|}\sum_{g\in G} \tr(g^2;\CC[K\backslash G]) = \sum_{V\in \hat G} \dim V^K \sigma(V).
	\end{equation}
	Comparing with~\eqref{eq:average-trace-1} gives the desired result.
\end{proof}
\begin{theorem}\label{theorem:easy-gelfand}
	Let $G$ be a finite group and $K$ be a subgroup.
	Then $KgK=Kg^{-1}K$ for all $g\in G$ if and only if the following conditions hold for every irreducible representation $V$ of $G$:
	\begin{enumerate}[1.]
		\item $\dim(V^K)\leq 1$.
		\item If $\dim(V^K)=1$, then $\sigma(V)=1$.
	\end{enumerate}
\end{theorem}
\begin{proof}
	By Theorems~\ref{theorem:dimension-hecke-algebra} and~\ref{theorem:trace-dagger-hecke}, $KgK=Kg^{-1}K$ for all $g\in G$ if and only if
	\begin{displaymath}
		\sum_{V\in \hat G} \dim(V^K)^2 = \sum_{V\in \hat G} \sigma(V)\dim(V^K).
	\end{displaymath}
	Such an equality holds if and only if $\dim(V^K)\leq 1$ for all $V\in \hat G$ and $\sigma(V)=1$ whenever $\dim(V^K)=1$.
	\end{proof}
\begin{example}
	Let $G=GL_n(\Fq)$ and, for each $1\leq k\leq n-1$, $K=P_k(\Fq)$ be the subgroup that fixes a $k$-dimensional subspace in $\Fq^n$.
	From the general theory of parabolic induction in $GL_n(\Fq)$, irreducible representations of $GL_n(\Fq)$ that have invariant vectors under $P_k(\Fq)$ for some $1\leq k\leq n-1$ are in bijection with partitions of $n$ with at most two parts.
	In this case $KgK=Kg^{-1}K$ for all $g\in G$, so every irreducible representation of $GL_n(\Fq)$ that admits a non-zero $P_k(\Fq)$-invariant vector for some $1\leq k\leq n-1$ has Frobenius-Schur indicator $1$.
\end{example}
\subsection{Application to Kronecker Coefficients}\label{section:application-kronecker}
Recall that $\rconj_d(G)$ denotes the set of real simultaneous conjugacy classes in $G^d$.
Applying Theorem~\ref{theorem:trace-dagger-hecke} to Kronecker-Hecke algebras gives
\begin{theorem}\label{theorem:trace-dagger-kron}
	Let $G$ be a finite group and $d$ be a positive integer.
	Then
	\begin{displaymath}
		|\rconj_d(G)| = \sum_{(V_1,\dotsc,V_{d+1})\in \hat G^{d+1}} \sigma(V_1)\dotsb \sigma(V_{d+1})\kappa(V_1,\dotsc,V_{d+1})
	\end{displaymath}
	is equal to the number of real simultaneous conjugacy classes of $d$-tuples in $G$.
\end{theorem}
For $d=1$, Theorem~\ref{theorem:trace-dagger-kron} is the well-known result that the number of real conjugacy classes in $G$ is equal to the number of isomorphism classes of self-dual irreducible representations of $G$.

Applying Theorem~\ref{theorem:easy-gelfand} to Kronecker-Hecke algebras gives a representation-theoretic characterization of $d$-real groups.
\begin{theorem}\label{theorem:d-real}
	Let $G$ be a finite group and $d$ be a positive integer.
	Then $G$ is $d$-real if and only if both the following conditions hold:
	\begin{enumerate}[1.]
		\item The Kronecker coefficient $\kappa(V_1,\dotsc,V_{d+1})\leq 1$ for all $(V_1,\dotsc,V_{d+1})\in \hat G^{d+1}$.
		\item If $\kappa(V_1,\dotsc,V_{d+1})=1$, then $\sigma(V_1)\dotsb \sigma(V_{d+1})=1$.
	\end{enumerate}
\end{theorem}
\subsection{Enumerating Simultaneous Real $d$-tuples}\label{section:enumeration}
For $g\in G$, let $r(g)$ denote the number of elements $x\in G$ such that $x^2=g$ (the number of square roots of $g$).
Frobenius and Schur~\cite{FS} showed that
\begin{equation}\label{eq:square-roots}
	r(g) = \sum_{U\in \hat G} \sigma(U)\chi_U(g).
\end{equation}
\begin{theorem}\label{theorem:rational-gf}
	Let $G$ be a finite group and $d$ be a positive integer.
	Then
	\begin{displaymath}
		|\rconj_d(G)| = \frac 1{|G|}\sum_{g\in G}r(g)^{d+1}.
	\end{displaymath}
\end{theorem}
\begin{proof}
	By Theorem~\ref{theorem:trace-dagger-kron},
	\begin{align*}
		|\rconj_d(G)| & = \sum_{(U_1,\dotsc,U_{d+1})\in \hat G^{d+1}} \sigma(U_1)\dotsb \sigma(U_{d+1})\kappa(U_1,\dotsc,U_{d+1})\\
		& = \sum_{(U_1,\dotsc,U_{d+1})\in \hat G^{d+1}} \sigma(U_1)\dotsb \sigma(U_{d+1}) \frac 1{|G|}\sum_{g\in G} \chi_{U_1}(g)\dotsb \chi_{U_{d+1}}(g)\\
		& = \frac 1{|G|}\sum_{g\in G} \Big(\sum_{U\in \hat G} \sigma(U)\chi_U(g)\Big)^{d+1}\\
		& = \frac 1{|G|}\sum_{g\in G} r(g)^{d+1} & \text{[by~\eqref{eq:square-roots}]},
	\end{align*}
	as claimed.
\end{proof}
Let $r_{\max}(G) = \max\{r(g)\mid g\in G\}$.
We have:
\begin{corollary}\label{corollary:exponential-growth}
	For any finite group $G$, the number of real simultaneous conjugacy classes of $d$-tuples in $G$ grows exponentially with exponential growth factor $r_{\max}(G)$.
\end{corollary}
\begin{example}
	When $G$ is the quaternion group $Q_8=\{\pm 1, \pm\ii, \pm\jj, \pm\kk\}$, $r_{\max}(G) = r(-1)=6$.
	In contrast, for groups where $\sigma(U)\geq 0$ for all $U\in \hat G$,~\eqref{eq:square-roots} implies $r_{\max}(G) = r(\mathrm{id}_G)$, the number of elements $g\in G$ with $g=g^{-1}$.
\end{example}
\section{Doubly Real Groups}\label{section:doubly-real}
\subsection{A representation-theoretic characterization}\label{section:representation-theoretic-characterization}
\begin{theorem}\label{theorem:higher-reality}
	If $G$ is $3$-real, then $G$ is an elementary Abelian $2$-group.
\end{theorem}
\begin{proof}
	If $G$ is $3$-real, then by Theorem~\ref{theorem:commutativity-kronecker-hecke}, $G$ has multiplicity-free $3$-fold tensor products.
	Hence, by Theorem~\ref{theorem:higher-hecke-commutative}, $G$ is Abelian.
	Since $G$ is real, it has to be an elementary Abelian $2$-group.
\end{proof}
If $G$ is $d$-real for $d\geq 3$, then $G$ is $3$-real, hence is Abelian.
Therefore, we will focus on the case $d=2$.
\begin{definition}
	A group $G$ is said to be \emph{doubly real} if it is $2$-real.
\end{definition}
Theorem~\ref{theorem:commutativity-kronecker-hecke} tells us that if a finite group $G$ is doubly real, then it has multiplicity-free tensor products.
In this section, we will prove that the converse is true for real groups using the following lemma.
\begin{lemma}\label{lemma:multiplcity}
	Let $G$ be a finite group, and let $U,V,W$ be irreducible self-dual representations of $G$.
	Assume that the multiplicity of $W$ in $U\otimes V$ is $1$.
	Then $\sigma(U)\sigma(V)=\sigma(W)$.
\end{lemma}
\begin{proof}
	The hypothesis implies that $U\otimes V = W\oplus X$, where $X$ does not contain any copies of $W$.
	Say that a bilinear form $(\text{-},\text{-})$ has sign $\epsilon$ if $(y,x)=\epsilon(x,y)$ for all $x,y$.
	The hypothesis implies that there exist non-degenerate bilinear forms on $U$ and $V$ that are $G$-invariant having signs $\sigma(U)$ and $\sigma(V)$ respectively.
	These induce a non-degenerate bilinear form $B$ on $U\otimes V$ with sign $\sigma(U)\sigma(V)$.
	The corresponding isomorphism $W\oplus X\to W'\oplus X'$ must take $W\to W'$ since $X$ does not contain $W$ so $X'$ does not contain $W'\cong W$.
	Thus, the restriction of $B$ to $W\times W$ is non-degenerate, hence a $G$-invariant non-degenerate bilinear form on $W$ with sign $\sigma(U)\sigma(V)$.
	It follows that $\sigma(W)=\sigma(U)\sigma(V)$.
\end{proof}
\begin{theorem}\label{theorem:real-multiplcity-free}
	A finite group $G$ is doubly real if and only if it is real and it has multiplicity-free tensor products.
\end{theorem}
\begin{proof}
	Theorem~\ref{theorem:d-real} tells us that if $G$ is doubly real, then it has multiplicity-free tensor products.
	For the converse, suppose a real group $G$ has multiplicity-free tensor products.
	By Lemma~\ref{lemma:multiplcity}, for all irreducible representations $U,V,W$ of $G$ such that $\kappa(U,V,W)=1$, we have $\sigma(U)\sigma(V)\sigma(W)=1$.
	Hence, by Theorem~\ref{theorem:d-real}, $G$ is doubly real.
\end{proof}
\subsection{Example: Generalized Dihedral Groups}
\begin{definition}
	Let $A$ be an Abelian group.
	The generalized dihedral group associated to $A$ is the semidirect product $D(A) = A\rtimes C_2$ where $C_2=\langle s\rangle$, a cyclic group of order $2$, acts on $A$ by $s\cdot a = a^{-1}$.
\end{definition}
\begin{remark}
	For every positive integer $n$, let $C_n$ denote the cyclic group of order $n$. Then $D(C_n)$ is the dihedral group of order $2n$.
	The generalized dihedral group $D(S^1)$ associated to the circle group $S^1$ is the group $O_2(\mathbf R)$.
\end{remark}
\begin{theorem}\label{theorem:generalized-dihedral}
	For every Abelian group $A$, the generalized dihedral group $D(A)$ is doubly real.
\end{theorem}
\begin{proof}
	Let $s\in D(A)$ denote the generator of the subgroup $C_2$.
	Every element of $D(A)$ is either in $A$ or in $sA$.
	If $u\in sA$, then $u=u^{-1}$, so $gug^{-1}=u^{-1}$ if $g=1$ or $g=u$.
	If $u\in A$, then $(sa)u(sa)^{-1}=u^{-1}$ for any $a\in A$.
	Thus, $D(A)$ is real.
	To see that $D(A)$ is doubly real, let $u,v\in D(A)$.
	We seek an element $g\in G$ such that $gug^{-1}=u^{-1}$ and $gvg^{-1}=v^{-1}$.
	\begin{enumerate}[1.]
		\item If $u\in sA$ and $v\in sA$, we may take $g=1$,
		\item If $u\in sA$ and $v\in A$, we may take $g=u$,
		\item If $u$ and $v$ are both in $A$, we can take $g=s$.
	\end{enumerate}
	Thus, $D(A)$ is doubly real.
\end{proof}
\begin{corollary}
	Finite generalized dihedral groups have multiplicity-free tensor products.
\end{corollary}
\subsection{Example: Heisenberg Groups}\label{section:heisenberg}
Let $\Fq$ denote a finite field with $q$ elements.
The Heisenberg group $H_n(\Fq)$ is the subgroup of $\GL_{n+2}(\Fq)$ given by
\begin{displaymath}
	H_n(\Fq) = \left\{\left(\begin{smallmatrix}
		1 & x & z\\
		0 & I_n & y\\
		0 & 0 & 1
	\end{smallmatrix}\right)\middle\vert x,y\in \Fq^n, z\in \Fq\right\}.
\end{displaymath}
Denoting the elements of $H_n(\Fq)$ by $(x,y;z)$, the multiplication rule becomes
\begin{displaymath}
	(x,y;z)\cdot (x',y';z') = (x+x',y+y';z+z'+x\cdot y').
\end{displaymath}
\begin{theorem}
	The Heisenberg group $H_n(\Fq)$ is doubly real if $q=2$.
\end{theorem}
\begin{proof}
	For $(a_i,b_i;c_i)\in H_n(\Fq)$, $i=1,2$, we seek $(x,y;z)\in H_n(\Fq)$ such that
	\begin{displaymath}
		(a_i,b_i;c_i)\cdot(x,y;z)\cdot(a_i,b_i;c_i) = (x,y;z) \text{ for } i=1,2,
	\end{displaymath}
	which (for characteristic $2$) results in the equation
	\begin{displaymath} 
		b_i\cdot x + a_i\cdot y = a_i\cdot b_i, \; i=1,2.
	\end{displaymath}
	This system of equations has a solution for all $a_1,a_2,b_1,b_2\in \Fq^n$ if $q=2$.
\end{proof}
\begin{remark}
	For $q>2$, Theorem~\ref{theorem:nilpotent} implies that $H_n(\Fq)$ does not have multiplicity-free tensor products, hence is not doubly real.
\end{remark}
\subsection{Example: Extraspecial $2$-Groups}\label{section:extraspecial}
Recall that an extraspecial $p$-group is a group $G$ such that $[G,G]=Z(G)$ is the cyclic group $C_p$ of order $p$, and $G/Z(G)$ is an elementary Abelian $p$-group.
Given extraspecial $p$-groups $G_1$ and $G_2$, their central product $G_1\circ G_2$ is the quotient of $G_1\times G_2$ by the normal subgroup $\{(z,z)\mid z\in C_p\}$ (we have fixed an identification of $Z(G_i)$ with $C_p$ for each $i$).
Let $D_8$ denote the dihedral group of order $8$, and $Q_8$ the quaternion group.
Every extraspecial $2$-group is a central product of copies of $D_8$ and $Q_8$~\cite[Theorem~5.2]{MR231903}.
It is easy to see that a central product of doubly real extraspecial $2$-groups is doubly real, as it is a quotient of a product of doubly real groups.
By Theorem~\ref{theorem:generalized-dihedral}, $D_8$ is doubly real.
The group $Q_8$ has the same character table as $D_8$, so it also has multiplicity-free tensor products and is real, hence is doubly real by Theorem~\ref{theorem:real-multiplcity-free}.
We conclude:
\begin{theorem}
	Every extraspecial $2$-group is doubly real.
\end{theorem}
\subsection{Example: Generalized Quaternion Groups}\label{section:generalized-quaternion}
Let $A$ be a finite Abelian group of order at least $4$ with a unique element $z$ of order $2$.
Let $C_4=\langle s\rangle$ denote the cyclic group of order $4$.
Let $C_4$ act on $A$ by $s\cdot a = a^{-1}$.
The element $z$ of $A$ lies in the centre of the semidirect product $A\rtimes C_4$.
The generalized quaternion group $Q(A)$ is defined as the quotient:
\begin{displaymath}
	Q(A) = (A\rtimes C_4)/\langle zs^2 \rangle.
\end{displaymath}
\begin{theorem}\label{theorem:gen-q}
	Let $A$ be a finite Abelian group of order at least $4$ with a unique element $z$ of order $2$.
	\begin{enumerate}[1.]
		\item $Q(A)$ has multiplicity-free tensor products.
		\item The following are equivalent:
		\begin{enumerate}[i.]
			\item $Q(A)$ is doubly real,
			\item $Q(A)$ is real,
			\item $z\in A^2$.
		\end{enumerate}
	\end{enumerate}
\end{theorem}
\begin{proof}
	The group $A$ has even order, say $2n$, and is a normal subgroup of $Q(A)$ with index $2$.
	The elements $1$ and $z=s^2$ form the centre of $Q(A)$.
	The remaining $2n-2$ elements of $A$ lie in conjugacy classes of the form $\{a,a^{-1}\}$ and have centralizers of order $2n$.
	The $2n$ elements of $Q(A)-A$ fall into two conjugacy classes, each of size $n$ and have centralizers of order $4$.
	In particular, $Q(A)$ has two central classes, $n-1$ classes of size $2$, and $2$ classes of size $n$; a total of $n+3$ conjugacy classes.

	By Burnside's lemma, for any finite group $G$ and any positive integer $d$,
	\begin{equation}\label{eq:burnside}
		|\conj_d(G)| = \frac 1{|G|}\sum_{g\in G} |C_G(g)|^d.
	\end{equation}
	Accordingly,
	\begin{displaymath}
		\frac 1{4n}[2\times (4n)^2 + (2n-2)\times(2n)^2 + 2n\times 4^2] = 2n^2 + 6n + 8.
	\end{displaymath}

	Since $Q(A)$ has an Abelian subgroup of index $2$, its irreducible representations have dimensions $1$ and $2$.
	Suppose $\alpha$ is the number of one dimensional representations, and $\beta$ is the number of two-dimensional representations, then $\alpha+\beta=n+3$ and $\alpha+4\beta=4n$.
	Hence, $\alpha=4$ and $\beta=n-1$.

	Given $(U,V,W)\in\widehat{Q(A)}$, let $\delta=(\dim U, \dim V, \dim W)$.
	Let's calculate the left hand side of the identity in Theorem~\ref{theorem:sum-of-squares}.
	If $\delta=(1,1,1)$, and $\kappa(U,V,W)>0$, then $W$ is completely determined by $U$ and $V$, and so there are $16$ triples with $\kappa(U,V,W)=1$.
	If $\delta=(2,1,2)$, then $U\otimes V$ is irreducible of dimension $2$, so again if $\kappa(U,V,W)>0$, then $W$ is completely determined by $U$ and $V$, and so there are $4(n-1)$ triples with $\kappa(U,V,W)=1$.
	By symmetry of Kronecker coefficients, each of $\delta=(1,2,2)$ and $(2,2,1)$ also contribute $4(n-1)$ to the sum.
	If $\delta=(2,1,1)$ then $\kappa(U,V,W)=0$, because $U\otimes V$ is irreducible of dimension $2$.
	By symmetry, $\kappa(U,V,W)=0$ for $\delta=(1,2,1)$ and $(1,1,2)$ as well.
	By Theorem~\ref{theorem:sum-of-squares}, we have:
	\begin{displaymath}
		2n^2 + 6n + 8 = 16 + 12(n-1) + \sum_{\delta = (2,2,2)} \kappa(U,V,W)^2,
	\end{displaymath}
	whence
	\begin{equation}\label{eq:222}
		\sum_{\delta=(2,2,2)}\kappa(U,V,W)^2 = 2(n-1)(n-2).
	\end{equation}
	The representation $X:= \bigoplus_{\dim U = \dim V =2} U\otimes V$ has dimension $4(n-1)^2$.
	We have
	\begin{displaymath}
		\dim X = \sum_{\dim U=\dim V=2} \kappa(U,V,W)\dim W.
	\end{displaymath}
	Since $\kappa(U,V,W)=1$ for $d=(2,2,1)$, one-dimensional representations of $Q(A)$ span a $4(n-1)$ dimensional subspace of $X$.
	Thus, two-dimensional representations of $Q(A)$ must span a subspace of dimension $4(n-1)^2 - 4(n-1) = 4(n-1)(n-2)$ in $X$.
	It follows that the sum of multiplicities of two-dimensional representations in $X$ is $2(n-1)(n-2)$.
	In other words
	\begin{displaymath}
		\sum_{\delta=(2,2,2)}\kappa(U,V,W) = 2(n-1)(n-2).
	\end{displaymath}
	Comparing with Equation~\eqref{eq:222}, we see that $\kappa(U,V,W)=1$ whenever $\delta=(2,2,2)$.
	It follows that $Q(A)$ has multiplicity-free tensor products.

	Since $sas^{-1}=a^{-1}$, every element of $A$ is conjugate to its inverse in $Q(A)$.
	For the remaining elements, observe that, for any $a\in A-\{1,z\}$, $(sa)^{-1}=sza$.
	For $b\in A$, $b(sa)b^{-1} = sab^{-2}$ and $(sb)sa(sb)^{-1} = sa^{-1}b^2$.
	Thus $sa$ is conjugate to its inverse (and hence $Q(A)$ is real) if and only if $z\in A^2$.
	By Theorem~\ref{theorem:real-multiplcity-free}, $Q(A)$ is doubly real if and only if it is real.
	\end{proof}
\section{Groups that do not have Multiplicity-Free Tensor Products}\label{sec:groups-that-dont}
\subsection{Symmetric and Alternating Groups}\label{section:alternating-groups}
When $G$ is the symmetric group $S_n$, irreducible representations $U$ and $V$ for which $U\otimes V$ has multiplicity-free tensor products were characterized by Bessenrodt and Bowman~\cite{MR3720803}.
For alternating groups, such pairs are classified by Homma~\cite{homma}.
The following theorem is easily deduced from their results.
For completeness, we outline a proof.
\begin{theorem}\label{theorem:symmetric_and_alternating}
	Then the symmetric group $S_n$ has multiplicity-free tensor products if and only if $n<5$.
	The alternating group $A_n$ has multiplicity-free tensor products if and only if $n< 4$.
\end{theorem}
\begin{proof}
	For each integer partition $\lambda$ of $n$, let $V_\lambda$ denote the irreducible representation of $S_n$ corresponding to $\lambda$.
	For $n\geq 5$, it turns out that $\kappa(V_{(n,2)}, V_{(n,1,1)}, V_{(n,1,1)})=2$.
	Hence, $S_n$ does not have multiplicity-free tensors for $n\geq 5$.
	The restrictions of these representations of $A_n$ remain irreducible for $n>5$, so $A_n$ does not have multiplicity-free tensor products for $n>5$.
	The remaining finitely many cases are easily checked by direct computation.
\end{proof}
\begin{remark}\label{remark:subgroup}
	The group $S_4$ has multiplicity-free tensor products, but its subgroup  $A_4$ does not.
	Thus, the property of having multiplicity-free tensor products is not inherited by subgroups.
\end{remark}
\subsection{A Combinatorial Lemma}
\begin{lemma}\label{lemma:combinatorial}
	Let $z,a,q$ be positive integers such that $z<a$ and $q>2$.
	Let $G$ be a finite group of order $aq$ such that
	\begin{enumerate}[1.]
		\item $G$ has $z$ elements in its centre, $a-z$ elements with centralizer of size $a$, and $a(q-1)$ elements with centralizer of size $zq$.
		\item $G$ has $zq$ irreducible representations of dimension $1$ and $(a-z)/q$ irreducible representations of dimension $q$.
	\end{enumerate}
	Then $G$ does not have multiplicity-free tensor products.
\end{lemma}
\begin{proof}
Burnside's lemma~\eqref{eq:burnside} tells us that
	\begin{align*}
		|\conj_2(G)| &= \frac 1{aq}[z\times a^2q^2 + (a-z)\times a^2 + a(q-1)\times z^2q^2]\\
		& = zaq + \frac{a(a-z)}q + z^2q(q-1).
	\end{align*}
	By Theorem~\ref{theorem:sum-of-squares}, $|\conj_2(G)| = \sum_{U,V,W\in \hat G^3}\kappa(U,V,W)^2$.
	We break this sum into parts depending on the dimension vector $\delta=(\dim U,\dim V,\dim W)$.
	If $\delta=(1,1,1)$, then $U$ and $V$ determine $W$ uniquely and $\kappa(U,V,W)=1$.
	There are $z^2q^2$ possibilities for $U$ and $V$.
	If $\delta=(q,1,q)$, then again, when $\kappa(U,V,W)>0$, $U$ and $V$ completely determine $W$ and $\kappa(U,V,W)=1$.
	There are $z(a-z)$ possibilities for $U$ and $V$.
	Since $\kappa(U,V,W)$ is symmetric in $U$, $V$, and $W$, the number of possibilities for $\delta=(1,q,q)$ and $(q,q,1)$ is also $z(a-z)$.
	When $\delta=(1,1,q)$, $\kappa(U,V,W)=0$, since $U\otimes V$ is one-dimensional.
	By symmetry, $\kappa(U,V,W)=0$ when $\delta=(q,1,1)$ or $(1,q,1)$.
	This leaves us with the contribution of terms for which $\delta=(q,q,q)$, which we denote by $X$.
	The identity of Theorem~\ref{theorem:sum-of-squares} becomes
	\begin{displaymath}
		zaq + \frac{a(a-z)}q + z^2q(q-1) = z^2q^2 + 3z(a-z) + X,
	\end{displaymath}
	whence we get
	\begin{equation}\label{eq:X}
		X = zaq + \frac{a(a-z)}q + z^2q(q-1) - z^2q^2 - 3z(a-z).
	\end{equation}
	Now suppose that $G$ has multiplicity-free tensor products.
	Then each term with $\delta=(q,q,q)$ contributes at most $1$ to the sum of squares.
	There are $(a-z)/q$ possibilities for each of $U$ and $V$.
	Since $U\otimes V$ has dimension $q^2$, there are at most $q$ possibilities for $W$ given $U$ and $V$.
	We get
	\begin{displaymath}
		X \leq (a-z)^2/q.
	\end{displaymath}
	Comparing with Equation~\eqref{eq:X}, we get
	\begin{displaymath}
		zaq + \frac{a(a-z)}q + z^2q(q-1) - z^2q^2 - 3z(a-z) \leq (a-z)^2/q,
	\end{displaymath}
	which simplifies to
	\begin{displaymath}
		\frac{z(a-z)}q[(q-2)(q-1)-1] \leq 0,
	\end{displaymath}
	contrary to $z<a$ and $q>2$.
	So $G$ cannot have multiplicity-free tensor products.
\end{proof}
\subsection{Nilpotent Groups}\label{sec:nilpotent}
In this section, we shall prove that a nilpotent group whose $p$-primary part is non-Abelian for some odd prime $p$ does not have multiplicity-free tensor products. We begin with the following lemma.
\begin{lemma}\label{lemma:SvN}
    Let $G$ be a non-Abelian finite group whose centre $Z$ contains its commutator subgroup $[G,G]$.
    Let $\chi:Z\to \mathbf C^*$ be a homomorphism that is injective on $[G,G]$.
    Then there exists a unique irreducible character of $G$ whose restriction to $Z$ is $\chi$-isotypic.
    This character has degree $\sqrt{|G/Z|}$.
\end{lemma}
\begin{proof}
    Let $A$ be a maximal Abelian subgroup of $G$.
    Since $[G,G]\subset Z\subset A$, $A$ is normal in $G$.
    Let $\tilde\chi$ be a character of $A$ that extends $\chi$.

    We claim that for any $g\in G\setminus A$, $\tilde\chi^g\neq \tilde\chi$.
    Indeed, if $g\in G\setminus A$, then by the maximality of $A$ among Abelian subgroups, there exists $a\in A$ such that $gag^{-1}a^{-1}\neq \mathrm{id}_G$.
    By the injectivity of $\chi$ on $[G,G]$, we have $1\neq \chi(gag^{-1}a^{-1})=\tilde\chi(gag^{-1}a^{-1})=\tilde\chi^g(a)\tilde\chi(a)^{-1}$, so $\tilde\chi^g\neq \tilde\chi$.
    Hence $\chi$ has $|G|/|A|$ distinct conjugates of the form $\tilde\chi^g$ for $g\in G$.

    By definition, $\Ind_A^G\tilde\chi$ is a character of $G$ of degree $|G|/|A|$ whose restriction to $A$ is the sum of the $|G|/|A|$ distinct conjugates of $\tilde\chi$, each occurring with multiplicity $1$, and $\Ind_A^G\tilde\chi=\Ind_A^G\tilde\chi^g$ for all $g\in G$.
    In particular, its restriction to $Z$ is $\chi$-isotypic.

    It follows from the above that the number of irreducible characters of $G$ whose restriction to $Z$ is $\chi$-isotypic is the number of $G$-orbits of characters of $A$ which extend $\chi$.
    Since the number of characters of $A$ which extend $\chi$ is $|A/Z|$, and the size of each $G$-orbit is $|G/A|$, the number of irreducible characters of $G$ whose restriction to $Z$ is $\chi$-isotypic is $|A/Z|/|G/A|=|A|^2/|Z||G|$.

    We will now show that $|A|^2=|Z||G|$, which will prove the lemma.
    Define $B:G/Z\times G/Z\to \mathbf C^*$ by $B(xZ,yZ)=\chi(xyx^{-1}y^{-1})$.
    Then $B$ is a multiplicative character in each variable; $B(xZ,xZ)=1$ for all $x\in G$; and for any $x\in G\setminus Z$, there exists $y\in G$ such that $B(xZ,yZ)\neq 1$. In other words, $B$ is a \emph{symplectic bicharacter}.

    It follows that the map $\phi:xZ\mapsto (yZ\mapsto B(xZ,yZ))$ induces an isomorphism $G/Z\tilde\to \widehat{G/Z}$, where $\widehat{G/Z}$ is the group of characters of $G/Z$.
    Upon restriction to $A/Z$, we obtain a surjective homomorphism $\overline\phi:G/Z\to \widehat{A/Z}$.
    Since $A$ is maximal Abelian, $\phi(xZ)$ is trivial on $A/Z$ if and only if $x\in A$, so $\overline\phi$ induces an isomorphism $G/A\tilde\to \widehat{A/Z}$, giving $|G/A|=|A/Z|$, whence $|A|^2=|Z||G|$.
\end{proof}
\begin{corollary}\label{corollary:SvN}
    Let $G$ be a finite non-Abelian group with an element $a$ of odd order in its centre such that $G/\langle a\rangle$ is Abelian.
    Then $G$ does not have multiplicity-free tensor products.
\end{corollary}
\begin{proof}
    Let $Z$ denote the centre of $G$.
    Since $G/\langle a\rangle$ is Abelian, $[G,G]\subset \langle a\rangle\subset Z$.
    Since $[G,G]$ is a subgroup of $\langle a\rangle$, it is cyclic of odd order, and hence there exists a character $\chi:Z\to \mathbf C^*$ that is injective on $[G,G]$, and $\chi^2$ is also injective on $[G,G]$.

    By Lemma~\ref{lemma:SvN}, there exist unique irreducible representations $U_\chi$ and $U_{\chi^2}$ of $G$ whose restrictions to $Z$ are $\chi$-isotypic and $\chi^2$-isotypic, respectively.
    Now, $U_\chi\otimes U_\chi$ is $\chi^2$-isotypic on $Z$ of dimension $|G/Z|$, so by Lemma~\ref{lemma:SvN}, it is a sum of $\sqrt{|G/Z|}$ copies of $U_{\chi^2}$.
    Therefore, $G$ does not have multiplicity-free tensor products.
\end{proof}
We are now ready to prove the main result of this section.
\begin{theorem}\label{theorem:nilpotent}
	Let $G$ be a finite nilpotent group.
	If there exists an odd prime $p$ such that the $p$-primary part of $G$ is non-Abelian, then $G$ does not have multiplicity-free tensor products.
\end{theorem}
\begin{proof}
	Every nilpotent group is a direct product of its $p$-primary parts.
	Therefore, it suffices to show that if $p$ is an odd prime, a non-Abelian $p$-group cannot have multiplicity-free tensor products. We will prove this by induction on the order of $G$.

	Assume that no non-Abelian $p$-group of order less than $p^k$ has multiplicity-free tensor products.
	Let $G$ be a non-Abelian $p$-group of order $p^k$.
	Let $a$ denote a central element of $G$ of order $p$.
	If $G/\gen{a}$ is non-Abelian, then it does not have multiplicity-free tensor products by the inductive hypothesis.
	Therefore, $G$ also does not have multiplicity-free tensor products.

	If $G/\gen{a}$ is Abelian, then $G$ does not have multiplicity-free tensor products by Corollary~\ref{corollary:SvN}.
\end{proof}
\subsection{Groups of Odd Order}\label{sec:odd-order}
We will show that no non-Abelian group of odd order has multiplicity-free tensor products.
For this we need a lemma about Frobenius groups.

Recall that a Frobenius group is a group of the form $G=P \rtimes H$, where $pHp^{-1}\cap H =\{\id\}$ for every $p\in P\setminus \{\id\}$.
In this case  we have (see~\cite[Lemma~7.3]{MR2270898}):
\begin{equation}\label{eq:frobenius-group}
	G\setminus P = \coprod_{x\in P} (xHx^{-1}- \{id\}).
\end{equation}
\begin{lemma}\label{lemma:frobenius}
	Let $G = P \rtimes H$ be a Frobenius group with the following properties:
	\begin{enumerate}[1.]
		\item $P$ is an elementary Abelian $p$-subgroup of order $p^b$,
		\item $H$ has order coprime to $p$,
		\item $P$ is the derived subgroup of $G$,
		\item $H$ is Abelian, and $H$ acts faithfully and irreducibly on $P$.
		\item $P$ is the unique minimal normal subgroup of $G$.
	\end{enumerate}
	Then $G$ does not have multiplicity-free tensor products unless $|H|=2$, in which case $G \cong D_{2p}$ is the dihedral group of order $2p$ and hence has multiplicity-free tensor products.
\end{lemma}
\begin{proof}
	We will prove the first part of the Lemma by showing that $G$ satisfies the hypotheses of Lemma~\ref{lemma:combinatorial} with $z=1$, $q=|H|$ and $a=p^b$.

	We first show that no non-trivial element of $h\in H$ commutes with any non-trivial element $x\in P$.
	Indeed, if $hx = xh$, then $x$ lies in the centre $Z$ of the subgroup generated by $h$ and $P$.
	Since $Z$ is a non-trivial normal subgroup of $G$, it contains $P$.
	Therefore, $h$ commutes with all the elements of $P$.
	It follows that the subgroup generated by $h$ is normal in $G$, contradicting the uniqueness of $P$ among minimal normal subgroups of $G$.
	Thus, the centralizer of each non-trivial element of $P$ is $P$, and the centralizer of each non-trivial element of $H$ is $H$.
	Since any element of $G\setminus P$ is conjugate to an element of $H$, by~\eqref{eq:frobenius-group}, its centralizer is a conjugate of $H$.
	It follows that every element of $G\setminus P$ has centralizer of size $|H|$.
	Hence $G$ satisfies the first hypothesis of Lemma~\ref{lemma:combinatorial}.

	Let $\chi$ be a non-trivial linear character of $P$.
	Let $I(\chi)=\{g\in G\mid \chi^g=\chi\}$ be its inertia subgroup.
	Suppose $h\in I(\chi)\cap H$.
	Then the kernel $\ker\chi$ of $\chi$ is invariant under the action of $h$ on $P\cong\mathbf{F}_p^b$.
	Since $h$ is $p$-regular, it acts on $\mathbf F_p^b$ by a semisimple matrix.
	Therefore, $\ker\chi$ has a one-dimensional $h$-invariant complement $D=\langle d\rangle$ in $P$.
	Suppose $hdh^{-1}=d^j$ for some $1\leq j<p$.
	Then $h\in I(\chi)$, $\chi(d)=\chi(hdh^{-1})=\chi(d^j)$.
	Since $\chi(d)$ is a non-trivial (hence primitive) $p$th root of unity, we must have $j=1$.
	Therefore $h$ commutes with a non-trivial element of $P$.
	As we have seen in the previous paragraph, we must have $h=\id$.
	Hence $I(\chi)=P$.

	Since the action of $H$ on $P$ is faithful, there are $(|P|-1)/|H|$ $H$-orbits of non-trivial linear characters of $P$.
	By Clifford theory, it follows that there are $(|P|-1)/|H|$ irreducible representations of dimension $|H|$.
	The remaining irreducible characters of $G$ are the $|H|$ one-dimensional characters extending the trivial character of $P$.
	Thus $G$ also satisfies the second hypothesis of Lemma~\ref{lemma:combinatorial}, and so it does not have multiplicity-free tensor products if $|H|>2$.

	Suppose that $|H|=2$ so $G = P \rtimes C_2$.
	Let $z$ denote the generator of $C_2$.
	Then $z$ acts on $P$ by an involution $\phi: P \to P$.
	Equivalently, $\phi$ is an $\mathbf F_p$-linear involution on $P$ of order $2$, so it has eigenvalues $\pm 1$.
	Since any $\phi$-invariant subgroup of $P$ is a normal subgroup of $G$, it must be trivial or all of $P$.
	It follows that $P$ is cyclic of order $p$ and $\phi$ acts on $P$ by either the identity or inversion.
	If $\phi$ acts as the identity, then $G$ is Abelian, contradicting the hypothesis that $G$ is non-Abelian.
	Therefore, $\phi$ acts as inversion, and $G$ is the dihedral group of order $2p$.
\end{proof}
\begin{theorem}\label{theorem:odd-order}
	If $G$ is a non-Abelian group of odd order, $G$ does not have multiplicity-free tensor products.
\end{theorem}
\begin{proof}
	Let $G$ be a non-Abelian group of odd order.
	We will show that $G$ does not have multiplicity-free tensor products using a series of reductions.
	By the Feit-Thompson theorem~\cite{MR166261}, $G$ is solvable.
	\begin{reduction}\label{reduction:Abelian-quotients}
		For any non-trivial normal subgroup $N$ of $G$, $G/N$ is Abelian.
	\end{reduction}
	\begin{proof}
		If not, replace $G$ by $G/N$ and proceed as before.
	\end{proof}
	\begin{reduction}\label{reduction:normal_quotient_abelian}
		For some prime $p$, $G$ has an elementary Abelian normal $p$-subgroup $P$ that is minimal among all non-trivial normal subgroups.
		The commutator subgroup $[G,G]=P$, and $P$ is the unique minimal normal subgroup of $G$.
	\end{reduction}
	\begin{proof}
	Let $P$ be a non-trivial minimal normal subgroup of $G$.
	Then $P$ must be Abelian since its derived group, being a proper characteristic subgroup of $P$, is normal in $G$, hence trivial.
	Let $p$ be a prime number dividing the order of $P$.
	The subgroup $\{a\in P\mid a^p = \mathrm{id}\}$ is a characteristic non-trivial elementary Abelian subgroup of $P$ hence normal in $G$, hence must be equal to $P$.
	Therefore $P$ is an elementary Abelian $p$-group.

	By Reduction~\ref{reduction:Abelian-quotients}, $G/P$ is Abelian so $\{\mathrm{id}\}\lneq [G,G]\leq P$.
	Since $[G,G]$ is normal, and $P$ is minimal normal, we must have $[G,G]=P$.
\end{proof}
\begin{reduction}\label{reduction:not-p}
	$|G/P|$ is coprime to $p$.
\end{reduction}
\begin{proof}
	If $G$ is a $p$-group, then $G$ cannot have multiplcity-free tensor products by Theorem~\ref{theorem:nilpotent}.
	Therefore, we may assume that $|G|$ is divisible by at least one more prime besides $p$, say $q$.
	Let $Q$ denote a Sylow $q$-subgroup of $G$.
	Since $G/P$ is Abelian, $PQ$ is a normal subgroup of $G$.

	Let the order of $Q$ be $q^\beta$ and the order of $P$ be $p^\alpha$.
	We claim that the number of Sylow $q$-subgroups of $PQ$ is $p^\alpha$.
	To prove this claim, we need to show that for any non-identity element $p\in P$, $pQp^{-1}\neq Q$.
	Let $K=\{x\in P\mid xQx^{-1}=Q\}$.
	For $x\in K$ and $y\in Q$, $xyx^{-1}y^{-1}\in P\cap Q = \{\id\}$.
	Therefore $K$ is contained in the centre of $PQ$.
	But the centre of $PQ$, being a characteristic subgroup of $PQ$, is normal in $G$.
	If $K$ is non-trivial, then by minimality of $P$, the centre of $PQ$ contains $P$, so $PQ\cong P\times Q$. 
	Thus $Q$, being a characteristic subgroup of $PQ$, is normal in $G$.
	By Reduction~\ref{reduction:normal_quotient_abelian}, $Q$ contains $P$, which is a contradiction.
	It follows that $K$ is trivial, so the number conjugates of $Q$ in $PQ$ is $p^\alpha$.

	Let $F$ be a Sylow $p$-subgroup of $G$.
	Again, since $FQ$ contains $P$ and $G/P$ is Abelian, $FQ$ is a normal subgroup of $G$.
	Since $PQ$ is normal, every conjugate of $Q$ in $G$ is contained in $PQ$, and hence is a Sylow $q$-subgroup of $PQ$.
	Therefore, the number of Sylow $q$-subgroups of $FQ$ is also $p^\alpha$.

	Therefore, $N=N_{FQ}(Q)$ has index $p^\alpha$ in $FQ$, so that
	\begin{equation}
		|N| = |FQ|/p^\alpha = |F|q^{\beta}/p^\alpha.
		\label{eq:n}
	\end{equation}
	We claim that
	\begin{equation}
		[F:N\cap F] = p^\alpha.
		\label{eq:normal-complement}
	\end{equation}
	Since $FQ\supset N\supset Q$, $FQ=FN$.
	We have
	\begin{equation}
		Q \cong FQ/F = FN/F \cong N/(N\cap F).
		\label{eq:normal-complement-2}
	\end{equation}
	Hence, using~\eqref{eq:n},
	\begin{equation}
		|N\cap F| = |N|/q^\beta = |F|/p^\alpha,
	\end{equation}
	from which~\eqref{eq:normal-complement} follows.
	In other words, $N\cap F$ provides a complement of $P$ in $F$.
	Since $P$ is normal in $G$ (and hence in $F$), $F=P\rtimes (N\cap F)$.

	Since $F$ contains $P$, $F$ is normal in $G$ by Reduction~2. The centre of $F$ is non-trivial and a characteristic subgroup of $F$, hence $Z(F)$ is a non-trivial normal subgroup of $G$. Therefore, $Z(F)$ contains $P$ by Reduction 2. In particular, $F=P\times (N \cap F)$.

	Since $N=N_{FQ}(Q)$, $Q$ is normal in $N$.
	Moreover, $N \cap F$ is normal in $N$ as $F$ is normal in $G$.
	Thus $Q$ and $N\cap F$ are normal subgroups of $N$ with trivial intersection.
	For any $x\in Q$ and $y\in N\cap F$, $xyx^{-1}y^{-1} \in Q \cap (N \cap F) = \{\id\}$, so $xy = yx$.
	Therefore, every element of $Q$ commutes with every element of $N\cap F$.

	From Reduction~2, $[G,G] = P$, so $[N \cap F, N\cap F] \subset P \cap (N \cap F) = \{\id\}$, so $N \cap F$ is Abelian.
	Since $F = P \times (N \cap F)$, every element of $N \cap F$ commutes with every element of $F$.
	Hence, $N \cap F$ is contained in the centre of $FQ$.
	If $N \cap F$ is non-trivial, then the centre of $FQ$, being a characteristic subgroup of $FQ$, is normal in $G$. Therefore, the centre of $FQ$ contains $P$ by Reduction~2. Hence, $x Q x^{-1} = Q$ for all $x \in P$, a contradiction.

	Hence $N\cap F$ is trivial, so $P=F$, so that $|G/P|$ is coprime to $p$.
\end{proof}
\begin{reduction}\label{reduction:semidirect}
	$G=P\rtimes H$ is a Frobenius group, where $P$ is an elementary Abelian $p$-group for some prime $p$, $H$ is an Abelian subgroup of $G$ such that $|H|$ is coprime to $p$, where $H$ acts faithfully and irreducibly on $P$, and $P$ is the unique minimal normal subgroup of $G$.
\end{reduction}
\begin{proof}
	Hall~\cite[Section~2]{MR1574393} showed that any solvable group $G$ of order $mn$, where $m$ and $n$ are coprime integers has at least one subgroup of order $n$.
	In the setting of Reduction~\ref{reduction:not-p}, $G$ has a complement $H$ to $P$.
	The kernel for the action of $H$ on $P$ is a central, and hence normal subgroup of $G$ that does not contain $P$, so it must be trivial.
	Therefore, $H$ acts faithfully on $P$.
	If $P_1\subset P$ is a proper $H$-invariant subgroup, then $P_1$ is a normal subgroup of $G$ that does not contain $P$, so $P_1$ is trivial.
	Therefore, $H$ acts on $P$ irreducibly.

    In order to show that $G$ is a Frobenius group, it suffices to show that $xHx^{-1} \cap H = \{id\}$ for all $x \in P \setminus \{id\}$.
    Suppose not, then there exist $x \in P \setminus \{id\}$ and $h\in H\setminus \{id\}$ such that $xhx^{-1} \in H$.
    We have $(xhx^{-1})h^{-1} \in H$.
    The normality of $P$ implies that $x(hx^{-1}h^{-1}) \in P$.
    Therefore, $xhx^{-1}h^{-1} \in P \cap H = \{id\}$, which implies $x \in C_G(h)$.
    
    For $u \in H$ and $g \in (P \cap C_G(h))$, we have $ugu^{-1} \in C_G(h)$. Indeed, since $H$ is abelian and $g \in C_G(h)$, we have $(ugu^{-1})h(ugu^{-1})^{-1}=h$.
    Therefore $P \cap C_G(h)$ is a non-trivial $H$-invariant subspace of $P$, so $P \cap C_G(h)=P$, which implies that $h$ is in the kernel of the action of $H$ on $P$, a contradiction.
    Hence, $H$ is a Frobenius complement of $G$.
\end{proof}
Since $|H|>2$, it follows from Lemma~\ref{lemma:frobenius} that $G$ does not have multiplicity-free tensor products.
This completes the proof of the theorem.
\end{proof}
\subsection{Non-Abelian Finite Simple Groups}\label{sec:lie-type}
\begin{theorem}\label{theorem:finite-simple}
	No non-Abelian finite simple group has multiplicity-free tensor products.
\end{theorem}

We have already seen that alternating groups $A_n$, $n\geq 5$ do not have multiplicity-free tensor products (Theorem~\ref{theorem:symmetric_and_alternating}).
The character table of the sporadic simple groups are available in the GAP character table library~\cite{GAP-CTBL}.
Using these character tables it is easy to see that the sporadic simple groups do not have multiplicity-free tensor products.
It only remains to consider the finite simple groups of Lie type.
\begin{theorem}[{\cite[Lemma~2.1]{MR1489911}}]\label{upperbound}
	If $G$ is a finite simple group of Lie type in characteristic $p$, then either
	\begin{enumerate}
		\item 
		$|\conj_1(G)|<|G|_p$ or
		\item 
		$G=PSL_2(q)$, $q$ even, and $|\conj_1(G)|=q+1$, or 
		\item 
		$G=PSL_2(5)$, and $|\conj_1(G)|=5$.
	\end{enumerate}
\end{theorem}

For a finite group $G$, let $d_G:=\max\{\dim V \mid V \in \hat G \}$ and $V_{\mathrm{max}}$ be an irreducible representation of dimension $d_G$.
\begin{lemma}\label{lem:nmftp}
	If a non-Abelian finite group $G$ satisfies $|\conj_1(G)| \leq d_G$, then $V_{\mathrm{max}} \otimes V_{\mathrm{max}}$ is not multiplicity-free. In particular, $G$ does not have multiplicity-free tensor products.
\end{lemma}
\begin{proof}
	Assume that $|\conj_1(G)| \leq d_G$. Since $d_G>1$ for any non-abelian group $G$, we have
	\[
	\sum_{V \in \hat G} \dim V < d_G |\hat G| \leq d_G^2. 
	\]
Therefore, $V_{\mathrm{max}} \otimes V_{\mathrm{max}}$ is not multiplicity-free.
\end{proof}
\begin{theorem}\label{theorem:lie-type}
No finite simple group of Lie type has multiplicity-free tensor products.
\end{theorem}

\begin{proof}
	Suppose $G$ is a finite simple group of Lie type other than $PSL_2(q)$, $q$ even and $PSL_2(5)$.
	Let $\St$ denote the Steinberg representation of $G$. For any $g\in G$, the value of the character $\chi$ of $\St$ is given by
	\begin{equation}\label{eq:Steinberg-character}
		\chi(g) = \pm |Z_G(g)|_p,
	\end{equation}
	the $p$-part of the cardinality of the centralizer of $g$.
	In particular, the Steinberg representation has dimension equal to the order $|G|_p$ of a Sylow $p$-subgroup of $G$.

	We have $d_G \geq \dim (\St)=|G|_p $. By Theorem~\ref{upperbound} we get  $|\conj_1(G)|<|G|_p \leq d_G$.
	
	When $G=PSL_2(q)$, $q$ even, we have $|\conj_1(G)|=q+1$ and $d_G=q+1$. For $G=PSL_2(5)$, $|\conj_1(G)|=5$ and $d_G=6$. 
	
	In all the above cases $|\conj_1(G)| \leq d_G$, so $G$ does not have multiplicity-free tensor products by Lemma~\ref{lem:nmftp}.
\end{proof}
\subsection{General Linear Groups}\label{sec:general-linear}
\begin{theorem}
	$GL(n,q)$ does not have multiplicity-free tensor products for $n\geq 2$ except when $n=2$ and $q=2$.
\end{theorem}
\begin{proof}
	Lettelier~\cite[Proposition~31]{Emmanuel_Letellier_2013} proved that $\langle \St\otimes \St, \St\rangle$ is a monic polynomial in $q$ of degree $(n-1)(n-2)/2$ with non-negative integer coefficients.
	If $n\geq 3$, then this polynomial has positive degree hence takes values that are greater than or equal to $2$, so $GL(n,q)$ is not multiplicity-free for $n\geq 3$.
	For $n=2$, the result follows from the explicit computation of Kronecker coefficients in~\cite{gupta2023tensor}.
\end{proof}

\subsection*{Note}
While this manuscript was under review, we became aware that Wigner~\cite{Wigner1941} had initiated the study of the class of real groups with multiplicity-free tensor products and called them \emph{simply reducible} groups.
Mackey studied such groups using double cosets and proved a result~\cite[Theorem~2]{Mackey1958} similar to our Theorem~\ref{theorem:easy-gelfand}.
He also showed that a $3$-real group is Abelian (Theorem~6 of~\cite{Mackey1958}), which is our observation at the beginning of Section~\ref{section:doubly-real}.
In Kazarin--Yanishevskii~\cite{KazarinYanishevskii2008} and Kazarin--Chankov~\cite{KazarinChankov2010}, it was shown that a finite group with multiplicity-free tensor squares is solvable, which implies our Theorem~\ref{theorem:finite-simple}.

\subsection*{Acknowledgements}
We thank A.~P.~Balachandran for bringing the work of Ben Geloun and Ramgoolam to our attention.
We thank Sanjaye Ramgoolam for sharing his work with us and some very helpful discussions.
We thank Dipendra Prasad and Steven Spallone for some insightful comments on the manuscript.
We thank Archita Gupta and Pooja Singla for help with GAP calculations.
We thank an anonymous referee for carefully reading our manuscript and providing several thoughtful suggestions that have improved the clarity of the paper.
\printbibliography
\end{document}